\documentclass[12pt]{amsart}
\usepackage{enumerate}
\usepackage{amssymb}
\usepackage{bbm}

\newtheorem{thm}{Theorem}

\newtheorem{prop}[thm]{Proposition}

\theoremstyle{definition}

\newcommand\0{\mathbf{0}}
\newcommand\1{\mathbbm{1}}
\newcommand\NN{\mathbb{N}}
\newcommand\RR{\mathbb{R}}

\newcommand\cB{\mathcal{B}}

\newcommand\cN{\mathcal{N}}
\newcommand\cP{\mathcal{P}}

\newcommand\pp{\mathbf{p}}

\renewcommand\iff{\Leftrightarrow}
\renewcommand\implies{\Rightarrow}

\newcommand\ceil[1]{\lceil{#1}\rceil}

\begin{document}%%%%%%%%%%%%%%%%%%%%%%%%%%%%%%%%%%%%%%%%%%%%%%%%%%%%%%
%%%%%%%%%%%%%%%%%%%%%%%%%%%%%%%%%%%%%%%%%%%%%%%%%%%%%%%%%%%%%%%%%%%%%%

\title{Erratum to:\\Lattice point methods for combinatorial games}
\author{Alan Guo}
\author{Ezra Miller}
\address{Department of Mathematics\\Duke University\\Durham, NC 27708}
\curraddr[Guo]{Electrical Engineering \& Computer Science\\MIT\\Cambridge, MA 02139}
\email{aguo@mit.edu, ezra@math.duke.edu}
\date{26 May 2011}

\maketitle

This note corrects Definition~6.3, Proposition~6.7, and Section~7, as
specified below.

Proposition~6.7 is false for the notion of squarefree game in
Definition~6.3, i.e., the equivalent conditions in Proposition~6.2.
Henceforth those equivalent conditions define \emph{weakly squarefree
games}.  For example, the game on $\NN^3$ with rule set
$\{(1,0,0),(0,1,0),(0,0,1),(1,1,0)\}$ is weakly squarefree but its
P-positions are easily shown not to satisfy Proposition~6.7.

The intended notion of \emph{squarefree game} is any game satisfying
the conditions in the following, whose parts mirror Proposition~6.2 as
closely as possible.

\begin{prop}
For a rule set $\Gamma$, the following are equivalent.
\begin{enumerate}
\item%
For each $\gamma \in \Gamma$, and $p,q \in \NN^d$, if $p + q - \gamma
\in \NN^d$ then $p - \gamma \in \NN^d$ or $q - \gamma \in \NN^d$.

\item%
(There is no appropriate analogue of condition~2 in Proposition~6.2.)

\item%
% The maximum entry of each $\gamma \in \Gamma$ is at most $1$, and
% there is at most one such entry.
Each $\gamma \in \Gamma$ has at most one positive entry, and that
entry is at most~$1$.

\item%
The positive part $\gamma_+$ is a 0-1 vector with at most one~$1$, for
all $\gamma \in \Gamma$.

\item%
Each move decreases the number of heaps of exactly one size, and the
amount of the decrease is~$1$.
\end{enumerate}
\end{prop}
\begin{proof}
1 $\implies$ 3: Fix $\gamma = (\gamma_1,\ldots,\gamma_d) \in \Gamma$.
First we show that the maximum entry of~$\gamma$ is at most $1$.  Let
$M = \max\{\gamma_1,\ldots,\gamma_d\}$, and let $p =
\ceil{\frac{M}{2}} \1$ where $\1$ is the vector with all entries equal
to $1$.  Then the minimum of the entries of $2p - \gamma$ is $1$ for
odd $M$ and $0$ for even $M$, and hence $2p - \gamma \in \NN^d$.
However, the minimum of the entries of $p - \gamma$ is
$\ceil{\frac{M}{2}} - M$ which is negative if $M > 1$.  But $p -
\gamma \in \NN^d$ by hypothesis, so $M \leq 1$.  Next we show that at
most one entry equals~~$1$.  Suppose that more than one entry is~$1$,
say $\gamma_i = \gamma_j = 1$ where $i \ne j$.  For each $k =
1,\ldots,d$, let $e_k$ be the $k$-th basis vector.  Let $p = \gamma
\vee \0 - e_j$
% $ = \sum_{k \ne j} \max\{\gamma_k, 0\} e_k$
and set $q = e_j$.  Then $p + q = \gamma \vee \0$ so $p + q - \gamma
\in \NN^d$, but $(p - \gamma)_j = -1$ and $(q - \gamma)_i = -1$.

3 $\implies$ 1: Fix $\gamma \in \Gamma$.  Then $\gamma$ has at least
one entry $\gamma_i = 1$, and therefore exactly one, because the real
cone $\RR_+\Gamma$ contains $\RR_+^d$ and is pointed.  If $p,q \in
\NN^d$ with $p + q - \gamma \in \NN^d$, then $p_i + q_i \geq 1$,
whence at least one of $p_i$ and $q_i$ is $\geq 1$.  Say $p_i \geq 1$.
Then $p - \gamma \in \NN^d$ because $\gamma_j \leq 0$ for all $j \neq
i$.

3 $\iff$ 4: This is straightforward.

4 $\iff$ 5: In the notation of Examples 2.1 and 2.5, condition~5 is
the translation of condition~4 into the language of heaps.
% 
% 1 $\implies$ 2: If $p + q'$ is an option of $p + q$, so $p+q' =
% p+q-\gamma$ for some $\gamma \in \Gamma$, then $q' = q - \gamma$ is an
% option of $q$.
% 
% 2 $\implies$ 1: Suppose $p + q - \gamma \in \NN^d$.  Let $q' = q -
% \gamma \in \NN^d$ \comment{how do you know $q' \in \NN^d$?}.  Then $p
% + q'$ is an option of $p + q$, so $q'$ is an option of $q$, hence $q -
% \gamma \in \NN^d$.
\end{proof}

Next we verify that Proposition~6.7 does indeed hold for the corrected
notion of squarefree game.

\begin{prop}
If $p \in \cB$, then $p \in \cP \iff p \cong \0$.
\end{prop}
\begin{proof}
Observe that $\0 \in \cP$; consequently, $p \cong \0 \implies p \in
\cP$.  Therefore it remains only to prove the other direction, namely
$p \in \cP \implies p \cong \0$.  This has two parts: $p + q \in \cP$
whenever $p,q \in \cP$, and $p + q \in \cN$ whenever $q \in \cN$.
However, the second follows from the first, because $q \in \cN
\implies q - \gamma \in \cP$ for some $\gamma \in \Gamma$, so adding
the two P-positions $p$ and $q - \gamma$ always yields another
P-position $p + q - \gamma$, whence $p + q \in \cN$.

To finish the proof, we show that $p,q \in \cP \implies p + q \in \cP$
by induction.  Clearly $p + q \in \cP$ if $p = q = \0$.  Therefore
assume $p,q \in \cP$ with $p \succ \0$ or $q \succ \0$, and
inductively assume that $\hat p \in \cP \implies \hat p + q \in \cP$
for all $\hat p \prec p$ and $p + \hat q \in \cP \implies p + \hat q
\in \cP$ for all $\hat q \prec q$.  Fix $\gamma \in \Gamma$ such that
$p + q - \gamma \in \NN^d$.  (Such a $\gamma$ exists by Proposition~1,
because the tangent cone axiom implies the existence of a move whose
only negative coordinate is $-1$ and occurs where $p + q$ is positive;
but even if no such $\gamma$ existed we would already be done anyway,
because then $p + q \in \cP$ by definition.)  By Proposition~1, either
$p - \gamma \in \NN^d$ or $q - \gamma \in \NN^d$.  Suppose $p - \gamma
\in \NN^d$.  Then $p - \gamma \in \cN$, so $p - \gamma - \gamma' \in
\cP$ for some $\gamma' \in \Gamma$.  By our induction hypothesis,
$(p-\gamma-\gamma') + q = (p+q-\gamma)-\gamma' \in \cP$, so
$p+q-\gamma \in \cN$.  If $q - \gamma \in \Gamma$, then a similar
argument still yields $p + q - \gamma \in \cN$.  Since $\gamma$ was
arbitrary, $p + q \in \cP$.
\end{proof}

The proof of the algorithm in Section~7 is incorrect, although the
existence of the algorithm is still true.  We offer a simpler proof
which obviates Section~7, and we also provide analysis of the
complexity of the algorithm.

\begin{thm}
There is an algorithm for computing $\cP$ for a squarefree game in
normal play that runs in $O(2^d|\Gamma|)$ time and requires $O(2^d)$
space.
\end{thm}

\begin{proof}
By Theorem~6.11 it suffices to compute $\cP_0 = \cP \cap \{0,1\}^d$.
Let the \emph{outcome} of a position $\pp \in \cB$ be P if $\pp$ is a
P-position, and N otherwise.  If we use \textsc{true} to encode
P-positions and \textsc{false} to encode N-positions, then the outcome
of a position $\pp \in \cB$ is the logical \textsc{nor} of its legal
options. Therefore, we can use a dynamic programming approach by
recursively computing the outcomes of all the positions in $\{0,1\}^d$
while storing the results in memory so that the outcome of each
position need only be computed once. Furthermore, if a legal option
$\pp'$ of $\pp$ lies outside $\{0,1\}^d$, then by Theorem 6.11 the
outcome of $\pp'$ is the same if we take its coordinates
modulo~$2$. Since there are $2^d$ positions to compute and each
position is computed exactly once, which requires looking at the
\textsc{nor} of at most $|\Gamma|$ outcomes, the algorithm runs in
$O(2^d|\Gamma|)$ time and requires $O(2^d)$ space.
\end{proof}

%%%%%%%%%%%%%%%%%%%%%%%%%%%%%%%%%%%%%%%%%%%%%%%%%%%%%%%%%%%%%%%%%%%%%%
\subsection*{Acknowledgement}%%%%%%%%%%%%%%%%%%%%%%%%%%%%%%%%%%%%%%%%%

We are grateful to Alex Fink for alerting us to these errors.

%%%%%%%%%%%%%%%%%%%%%%%%%%%%%%%%%%%%%%%%%%%%%%%%%%%%%%%%%%%%%%%%%%%%%%

%%%%%%%%%%%%%%%%%%%%%%%%%%%%%%%%%%%%%%%%%%%%%%%%%%%%%%%%%%%%%%%%%%%%%%
\end{document}